\documentclass{amsart}

\usepackage{amsmath,amssymb,url}
\usepackage{enumitem} 

\numberwithin{equation}{section}

\theoremstyle{plain}
\newtheorem{theorem}{Theorem}
\newtheorem{lemma}[theorem]{Lemma}

\theoremstyle{definition}

\DeclareMathOperator{\Mod}{Mod}

\newcommand{\M}[1]{\operatorname{M}(#1)}
\def\FF{\mathcal{F}}
\def\Pa{\mathsf{Pa}}
\def\qb{\mathbf{qb}}
\def\BB{\bar{\mathbf{B}}}
\def\ppto{\twoheadrightarrow}
\def\F{\mathbf{F}}
\def\A{\mathbf{A}}

\begin{document}


\title[Two questions of K--S]{Two questions of Kowalski--Słomczy\'nska}



\author[Z. Gyenis]{Zal\'an Gyenis}
\address{Jagiellonian University \\
	Grodzka 52 \\
	31-007 Krak\'ow \\
	Poland}
\urladdr{}
\email{zalan.gyenis@uj.edu.pl}

\subjclass{06D15, 08B20, 06D20, 08B15}

\keywords{Distributive p-algebras, Quasivarieties, Implication-free fragment of intuitionism.}

\begin{abstract}
    That the inclusion $Q(F_3(X))\subseteq \Pa^-_3$ is proper and
    $\Mod(\qb_3)$ is not generated by free $p$-algebras is proved, 
    answering two open questions from Kowalski--Słomczyńska \cite{KowSlom}.
\end{abstract}

\maketitle

\noindent Following the framework and notation of Kowalski and Słomczyńska [1], we omit formal definitions here for the sake of brevity. All terminology used in this note is introduced in their original work.

Let $P$ be the poset of height $2$ with four maximal elements $\{a,b,c,d\}$ and minimal elements $t_{abc}$, $t_{abd}$, $t_{acd}$, and $t_{bcd}$ with order relations $t_{uvw}\le u,v,w$ and no other elements compared. Let $\A = \varepsilon(P)$ be the dual algebra.
For an element $x$ of a poset, $\M{x}$ denotes the set of maximal elements above $x$.

\begin{lemma}\label{lemma:1}
    $\A\in \Pa_3$.
\end{lemma}
\begin{proof}
    As $\Pa_3 = V(\BB_3)$ it is enough to show that $A$ embeds into $\BB_3^4$. By duality, this amounts to constructing
    a pp-morphism $\biguplus_{i=1}^{4}\delta(\BB_3)\ppto P$.
    For each minimal element $t$ of $P$ let $S_t$ 
    be a copy of $\delta(\BB_3)$. $S_t$ has $3$ maximal elements and one bottom element below all three. Define a map $g:\uplus_{t} S_t \twoheadrightarrow P$
    as follows: on the copy $S_{t_{abc}}$ map its three 
    maximal elements bijectively to $a,b,c$, and its bottom 
    element to $t_{abc}$; similarly for the other $S_t$'s.
    Then $g$ is a surjective pp-morphism: it is clearly order-preserving, and for each bottom element we have that the image of its three maximal extensions is exactly the set of maximal extensions of the corresponding $t_{abc}$.
\end{proof}

\begin{lemma}\label{lemma:2}
    $\A\notin Q(\mathcal{F})$, where $\mathcal F$ is the class of free $p$-algebras. In particular, $\A\notin Q(\F_3(X))$ for any $X$.
\end{lemma}
\begin{proof}
    For a contradiction, suppose $\A\in Q(\mathcal{F})$.     
    Because $\A$ is finite, \cite[Lemma 1.6]{KowSlom} gives $\A\in\mathbf{ISP}_{fin}(\mathcal{F})$, and so
    $\A\leq \Pi_{i}{\FF_i(X_i)}$, where $\FF_i(X_i)$ is a free $p$-algebra generated by the set $X_i$. Let $\pi_i:\A\to \FF_i(X_i)$ 
    be the coordinate  projections. Since $\A$ is finite, each image $\pi_i(\A)$ is a finite subalgebra of $\FF_i(X_i)$, hence it is contained in some finitely generated $\FF_i(k_i)$. 
    By \cite[Lemma 2.5]{KowSlom} finitely generated free $p$-algebras are of the form $\F_{m_i}(k_i)$, thus $\A\leq \prod_i\F_{m_i}(k_i)$. 
    By duality \cite[Lemma 1.7]{KowSlom} there 
    is a surjective pp-morphism 
    \[
        g:\; \biguplus_i P_i \twoheadrightarrow P,
        \qquad\text{where } P_i=\delta(F_{m_i}(k_i)).
    \]
    Pick the point $t_{abc}\in P$ and choose $i$ and 
    $y\in P_i$ such that $g(y)=t_{abc}$. The pp-condition of $g$
    gives $\M{g(y)} = g(\M{y})$.
    But $\M{g(y)}=\M{t_{abc}}=\{a,b,c\}$, so there 
    exist maximal elements
    $\alpha,\beta,\gamma\in \M{y}$ such that
        $g(\alpha)=a$, $g(\beta)=b$, $g(\gamma)=c$.
    It follows from \cite[Lemma 2.5(1)]{KowSlom}) that $k_i\geq 2$.
    Also, $m_i\geq 2$, because in $\F_1(k_i)$ there is no element $y$ 
    with $|\M{y}|=3$. Now apply \cite[Lemma~2.5(2)]{KowSlom} inside 
    $P_i$: every nonempty set of maxima of 
    size $\le m_i$ occurs as $\M{z}$ for some $z\in P_i$. In particular, there exists $z\in P_i$ such that $\M{z} = \{\alpha,\beta\}$.
    By the pp-condition it follows that
    \[
        \M{g(z)} = g(\M{z}) = \{g(\alpha),g(\beta)\}=\{a,b\}.
    \]
    But this is impossible in $P$: 
    there is no $x\in P$ with $\M{x}=\{a,b\}$.
\end{proof}

\begin{lemma}\label{lemma:3}
    $\A\in \Mod(\qb_3)$.
\end{lemma}
\begin{proof}
    By duality, $\A\models \qb_3$ iff there is no surjective pp-morphism $P\twoheadrightarrow \delta(\BB_3)$ (see \cite[Lemma 2.3]{KowSlom}). By way of contradiction, suppose 
    $f:P\twoheadrightarrow \delta(\BB_3)$ is a surjective 
    pp-morphism. Since $P$ has four maximal elements and $\delta(\BB_3)$ has three maximal elements, it follows that two distinct maximal elements, say $a,b$, must satisfy
    $f(a)=f(b)$. Since $f$ is surjective on maximal elements, pick a third maximal element, say $c$, with $f(c)\neq f(a)$. Consider $t_{abc}$. Then 
    $\M{t_{abc}}=\{a,b,c\}$ and by the pp-condition
    \[
        \M{f(t_{abc})}=f(\M{t_{abc}})=\{f(a),f(b),f(c)\}=\{f(a),f(c)\}
    \]
    is a $2$-element set. But in $\delta(\BB_3)$ there is no element $x$ with $|\M{x}|=2$. Hence no such $f$ exists and
    $A\models \qb_3$.
\end{proof}

\noindent Recall $\Pa^-_3 = \Pa_{3}\cap \mathrm{Mod}(qb_3)$.

\begin{theorem}
    The inclusion $Q(F_3(X))\subseteq \Pa^-_3$ is proper, and
    $\Mod(\qb_3)$ is not generated by free $p$-algebras.
\end{theorem}
\begin{proof}
    Immediate from the lemmas.
\end{proof}


\begin{thebibliography}{1}
\providecommand{\url}[1]{{#1}}
\providecommand{\urlprefix}{URL }
\expandafter\ifx\csname urlstyle\endcsname\relax
  \providecommand{\doi}[1]{DOI~\discretionary{}{}{}#1}\else
  \providecommand{\doi}{DOI~\discretionary{}{}{}\begingroup
  \urlstyle{rm}\Url}\fi

\bibitem{KowSlom}
Kowalski, T., Słomczyńska, K.: Quasivarieties of p-algebras: Some new
  results.
\newblock Studia Logica  (2025).
\newblock \doi{10.1007/s11225-025-10187-9}.
\newblock \urlprefix\url{https://doi.org/10.1007/s11225-025-10187-9}

\end{thebibliography}
\end{document}